\documentclass[english,a4paper,11pt]{article}
\usepackage{lmodern}

\usepackage{amssymb,amsmath,amsfonts,amsthm,epsfig,tikz,amscd} 
\usepackage{authblk}   
\usepackage{hyperref}
\usepackage{cite}

\usepackage[width=0.85\textwidth ]{caption}

\newtheorem{theorem}{Theorem}[section]
\newtheorem{corollary}{Corollary}[theorem]
\newtheorem{lemma}[theorem]{Lemma}
\title{Geometric properties of integrable Kepler and Hooke billiards with conic section boundaries}

\author[1]{Daniel Jaud}
\affil[1]{\footnotesize Gymnasium Holzkirchen, Germany, Daniel.Jaud.PhD@gmail.com}

\author[2]{Lei Zhao}
\affil[2]{Institute of Mathematics, University of Augsburg, Lei.Zhao@math.uni-augsburg.de}

\date{}

\begin{document}
\maketitle
\flushbottom


\begin{abstract}
\noindent
We study the geometry of reflection of a massive point-like particle at conic section boundaries. Thereby the particle is subjected to a central force associated with either a Kepler or Hooke potential. The conic section is assumed to have a focus at the Kepler center, or have its center at the Hookian center respectively. When the particle hits the boundary it is ideally reflected according to the law of reflection. These systems are known to be integrable.

We describe the consecutive billiard orbits in terms of their foci. We show that the second foci of these orbits always lie on a circle in the Kepler case.  In the Hooke case, we show that the foci of the orbits lie on a Cassini oval. {For both systems we analyze the envelope of the directrices of the orbits} as well.

\end{abstract}

\vspace*{0.4cm}

\vspace*{0.4cm}
{\textbf{MSC-Classification:} 14H70, 37C79, 37J99, 37N05}


\section{Introduction}

In 1868, Ludwig Boltzmann published an article \cite{Boltzmann1868} in which he thought established the ergodicity for a class of simple mechanical systems. The system he considered in particular contains the case of a point-like massive particle under the influence of a Keplerian $1/r-$potential and an ideally straight reflection wall. The orbit of the particle are arcs of Keplerian orbits.
When hitting the reflection wall, it gets elastically reflected. As we see this defines a Kepler billiard system.
In contrast to what Boltzmann asserted, it was recently shown in \cite{gallavotti2020theorem} that this system inhabits an additional integral of motion, is therefore integrable and non-ergodic. We next state this result:

\begin{theorem}[Gallavotti-Jauslin]
\label{thm:Galavotti}
Let $F$ and $F_i$ be the foci of a Kepler ellipse $\mathcal{K}_i$, where the center of force lies at $F$. Further consider a straight line $l$ which the orbit of the Kepler ellipse $\mathcal{K}_i$ intersects at the point of reflection $P$. Then the focus $F_{i+1}$ of the new reflected ellipse $\mathcal{K}_{i+1}$ lies on a common circle around the mirror image $F_{mr}$ of $F$ with respect to the line $l$ with radius $|F_i F_{mr}|$. 
\end{theorem}

This additional first integral allowed G. Felder \cite{Felder_2021} to identify an associated elliptic curve for the compactification of the joint level set, an integration of the problem with a short computation on $3$-periodic orbits. The computation on n-periodic orbits of the system has been continued by S. Gasiorek and M. Radnovi\'{c} \cite{gasiorek2023periodic}.
In \cite{Zhao2022}, this additional first integral is linked to the energy of a Kepler problem on the sphere.

In a related situation, it has been shown \cite{jaud2022gravitational} that the description of consecutive orbits in terms of the foci of a particle parabolic trajectory under the influence of a constant, uniform gravitational field, yields a fruitful analysis on the geometric properties of the system. 

In dimension 2, the Kepler problem is dual to the Hooke problem via the complex square mapping \cite{MacLaurin,Goursat}. Panov \cite{Panov} observed that this extends to Kepler and Hooke billiards. An alternative way to establish the first integral of Gallavotti-Jauslin is presented in \cite{takeuchi2021conformal}. The Hooke billiards with a centered conic section domain is integrable, which follows from \cite{Jacobi1866}. Their dynamics has been studied in e.g. \cite{Fedorov2001} or 

\noindent
\cite{Barrera2023}.


Apart from the two distinct systems it also appears in celestial mechanics \cite{Barutello_2023} that combinations of the Hooke and Kepler system give rise to a reasonable approach to describe the dynamics of a massive object within and outside a region of a given mass distribution, in terms of refraction billiard systems they introduce. Their method covers the case of a Kepler billiard in an elliptic domain, and they showed that in the case the Kepler center is put at the center of the elliptic domain, then the Kepler billiard system is chaotic. In \cite{deblasi2023}, Barutello and de Blasi explained that putting a Kepler center in a generic position inside an ellipse results always in a chaotic billiard system.


The main purpose of this work is to consider some geometric properties of the second foci in other integrable Hooke and Kepler billiards reflected at a conic section in the plane.

The structure of this paper is as follows: In Section \ref{sec:Setup} we introduce necessary concepts and general settings that will be used throughout this paper. 
In Section \ref{sec:Kepler}, we discuss the geometric interpretation of the additional first integral of the Kepler billiards from \cite{takeuchi2021conformal}, and describe the construction procedure of consecutive Kepler orbits in terms of consecutive (second) foci, which in general can be applied to Kepler billiards with other reflection walls. 
In addition we will derive some geometric properties of the system such as the focal reflection property as well as the envelope curves of the directrices of consecutive Kepler orbits. In Section \ref{sec:Hooke} we apply the conformal duality between the Hooke and Kepler system to carry the discussions to the Hooke billiards.
 
In particular, we show that in the corresponding Hooke billiards, consecutive foci of the Hooke orbits lie on certain Cassini ovals. We have thus identified a simple geometric property for which the Kepler side is simpler than the Hooke counterpart.


\section{Setup}\label{sec:Setup}


We consider the motion in a plane of a massive, point-like particle under the influence of either a Kepler or Hooke potential from the origin of the coordinate system. The Hamiltonians of the two systems are given respectively by
$$H_K=\frac{p_x^2+p_y^2}{2m}+\frac{\alpha}{\sqrt{x^2+y^2}},$$
and
$$H_H=\frac{p_x^2+p_y^2}{2m}+\frac{k}{2}\cdot \left(x^2+y^2\right).$$
Depending on the sign of the coupling constants $\alpha$ and $k$ and the energy, the orbits of the particle are either ellipses or hyperbolic arcs, or parabolas in the case of zero-energy in the Kepler case. We shall mainly focus on elliptic orbits in this study. The extension to other type of orbits are straightforward. For the elliptical orbits we will also use the term flight ellipse to indicate the ellipse containing the orbit arc 
outside the boundary. We will describe the corresponding flight ellipses in terms of their foci 
and semi-major axis $a$.
The particle is further assumed to reflect ideally according to the classical law of reflection along a conic section boundary $\mathcal{B}$, also called mirror, and otherwise propagates under the influence of the potential. We consider them to be oriented clock-wise if they are not degenerated to a line segment to have a well-defined billiard dynamics. We will refer to these dynamical systems as Kepler billiards and Hooke billiards respectively. 

We consider the Kepler billiard with the mirror boundary given by the \emph{focused} conic section

$${\mathcal{B}}_K:\frac{(x-\bf{c}_K)^2}{\bf{a}_K^2}+\frac{y^2}{\bf{a}_K^2 -\bf{c}_K^2}= 1.$$
In the non-degenerate cases the foci of such a conic lie at $F:=(0,0)$ and $F':=(2\bf{c}_K,0)$. The semimajor axis of $\mathcal{B}_K$ is $\bf{a}_K$. The linear eccentricity $\bf{c}_K$, which is the product of $\bf{a}_K$ with the eccentricity $e_K$, equals to the half distance between the two foci: $2 \bf{c}_K=|FF'|$. 

In order to distinguish the boundary parameters from the dynamical flight orbits we use bold letters for the boundaries parameters.
The label $\bf{K}$ corresponds to the system with the Kepler potential whereas the label $\bf{H}$ associated with the Hooke potential will be used later on.

For $\bf{c}_K<a_K$ we obtain ellipses as a boundary, for $\bf{c}_K>a_K$ hyperbolic arcs. Our main focus will lie on those two boundary configurations. On the other hand, many of our analysis extends to circular or parabolic boundaries by taking necessary limits.


For the Hooke system, we consider the boundary given by a \emph{centered} conic section
$${\mathcal{B}}_H:\frac{x^2}{{\bf{a}_H}^2}+\frac{y^2}{{\bf{a}_H}^2-{\bf{c}_H}^2}= 1,$$
where again depending on the value of the semi-major axis $\bf{a}_H$ with respect to the linear eccentricity $\bf{c}_H$ we obtain an elliptic or hyperbolic boundary.

In contrast to the Kepler system the foci ${F_H:=(-{\bf{c}_H},0)}$ and ${F_H':=({\bf{c}_H},0)}$ of the Hooke boundary $\mathcal{B}_H$ lie symmetrically around the origin, at which the center of the force locates. 

The motion of the associated ellipses for the Kepler or Hooke system can in principle lie on either side of the corresponding boundaries, with closely-related dynamics. We mainly focus on the region not containing the origin. Some representative cases are displayed in Fig. \ref{fig:Setup}. 

For the cases where the orbits lie in the domain containing the origin, the Hooke center is regular but the Kepler center is singular. Nevertheless, it is classically known that collisions with the Kepler center can be regularized by exactly the Hooke-Kepler correspondence which we shall use. With the regularization, the orbits continue backward after bounced back at the Kepler center. This consideration allows us to establish most of of our analysis in these cases as well.

\begin{figure}[htb]
    \centering
    \includegraphics[scale=0.8]{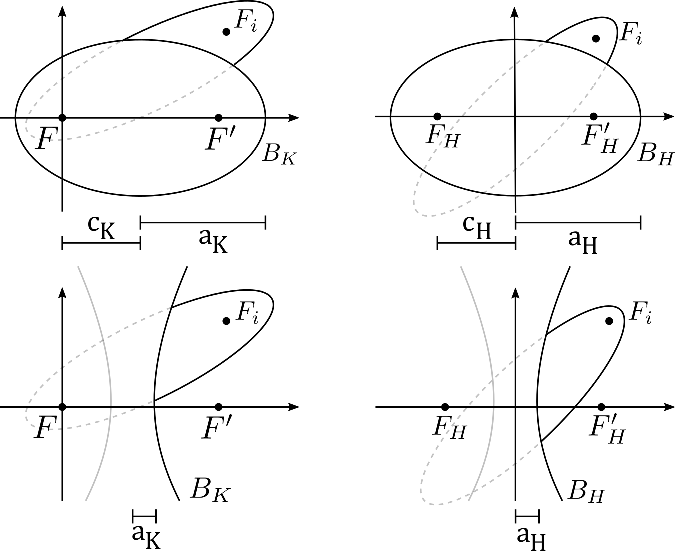}
    \caption{\textit{Left:} Representations of the considered cases for the Kepler system. \textit{Right:} Representation of the considered cases for the Hooke system.}
    \label{fig:Setup}
\end{figure}

In all systems, the consecutive \emph{flight ellipses} are characterized in terms of the energy and their consecutive foci. In the Kepler case, with the origin as a focus, the flight ellipses are completely determined by the positions of the \emph{second foci} $F_i$ and the conserved energy $E$. In the Hooke system the pair of foci $\{F_i, -F_i\}$ as well as the energy $E$ determine the form of the flight ellipse. We denote the consecutive orbits in the Kepler case by $\mathcal{K}_i$ and for the Hooke system by $\mathcal{H}_i$.

These mechanical billiard systems are known to be integrable in the sense that there exists an additional first integral independent of the energy \cite{Jacobi1866, Fedorov2001, Panov,takeuchi2021conformal, takeuchi2022projective}. In the next section we discuss this additional first integral and its geometrical implications in the Kepler case.

\section{Kepler billiards}\label{sec:Kepler}
\subsection{First integrals and envelopes}

In this section we derive the first integral in addition to the total energy for the Kepler billiard we set up.
We first present a  geometrical proof. Then we show that this agrees with the geometrical interpretation of the first integral established in \cite{takeuchi2021conformal}.

As already mentioned in Section \ref{sec:Setup} the associated consecutive Kepler flight ellipses, denoted by $\mathcal{K}_i$ with $i\in \mathbb{N}$, can be describe by the position of the second focus $F_i$ and the total energy $E<0$ of the particle. Since we have $E \sim 1/a$ from the study of Kepler problem, geometrically, the conservation of the total energy easily translates into the conservation of the semi-major axis $a$ for all $ \mathcal{K}_i$.
We will denote the second focus of $\mathcal{K}_i$ by $F_i=(F_{x,i},F_{y,i})$.

\begin{lemma}
\label{lemma:Circle_P}
Let $P$ be the point of reflection along the boundary $\mathcal{B}_K$ of a Kepler ellipse $\mathcal{K}_i$ with foci $F$ and $F_i$,
Then the 
second focus $F_{i+1}$ of the reflected ellipse $\mathcal{K}_{i+1}$ satisfies
$|PF_i|=|PF_{i+1}|$. 
\end{lemma}

This is a direct consequence of the definition for an ellipse and the preservation of the semi-major axis $a$ of the flight ellipses.

Having introduced all relevant definitions for the Kepler billiard we now state the first theorem describing a common curve on which the consecutive foci $F_i$ lie on. In \cite{jaud2022gravitational} these curves were coined \textit{foci curves}. Here we will adopt this notion.

\begin{theorem}[The foci circle in the Kepler case]\label{thm:Kepler_first_integral}
Consider the Kepler billiard with Kepler center $F$ with orbits outside a conic section boundary $\mathcal{B}_K$ with foci $F$ and $F'$. 
Let $\mathcal{K}_1$ be a Kepler ellipse (hyperbola) with foci $F$ and $F_1$. Then all the second foci $F_i$ of consecutive flight Kepler ellipses (hyperbolas) $\mathcal{K}_i$ lie on a common circle around the second focus $F'$ of $\mathcal{B}_K$ with radius $R=|F'F_i|=const.$. In particular, $R$ is a first integral in addition to the energy of this system.
\end{theorem}

The theorem gives a geometric interpretation of the first integral derived in \cite{takeuchi2021conformal}.
Note that in the case of a circular boundary the two foci $F$ and $F'$ coincide. Also the parabolic boundary may be viewed as the limiting case of a hyperbolic boundary where the distance between the foci $F$ and $F'$ is very large. In this case one effectively can produce a particle bouncing inside a parabolic mirror with constant gravitational force. For this limiting case the theorem reproduces the same result obtained in \cite{jaud2022gravitational}. Alternatively the limit ${\bf{a}_K}\rightarrow 0$ reproduces the theorem of Gallavotti-Jauslin \ref{thm:Galavotti} for the reflection along a straight line.
For the proof of this Theorem \ref{thm:Kepler_first_integral} we only consider the elliptic boundary case. The result for the hyperbolic boundary then follows by the same arguments. The equivalent results for the circular and parabolic boundary follow by the aforementioned limits.

\begin{proof}
Consider Fig. \ref{fig:Ellipse_Foci} \textit{Bottom}. Without loss of generality let the initial flight ellipse with second $F_1$ intersect $\mathcal{B}_K$ at the point of reflection $P\in \mathcal{B}_K$. 

We apply Thm. \ref{thm:Galavotti} with respect to  the tangent line of $\mathcal{B}_K$ at $P$. 
Since $|FP|=|F_PP|$ and, due to the optic property of the ellipse,
we see that $F_P$ lies on the straight line trough $P$ and $F'$. Applying Lemma \ref{lemma:Circle_P} and Thm. \ref{thm:Galavotti} we obtain $F_2$ as the intersection point of the circle centered at $F_P$ with radius $|F_PF_1|$
and the other circle centered at $P$ with radius $|PF_1|$.

Thus $F_2$ is the mirror point of $F_1$ with respect to the line $\overline{F'P}$, and thus $|F_1F'|=|F_2F'|$. Since the same construction holds true for all consecutive flight ellipses, it follows that $|F'F_i|=|F'F_1|=R=const.~\forall i\in \mathbb{N}$.
\end{proof}

\begin{figure}[!htb]
    \centering
    \includegraphics[scale=0.5]{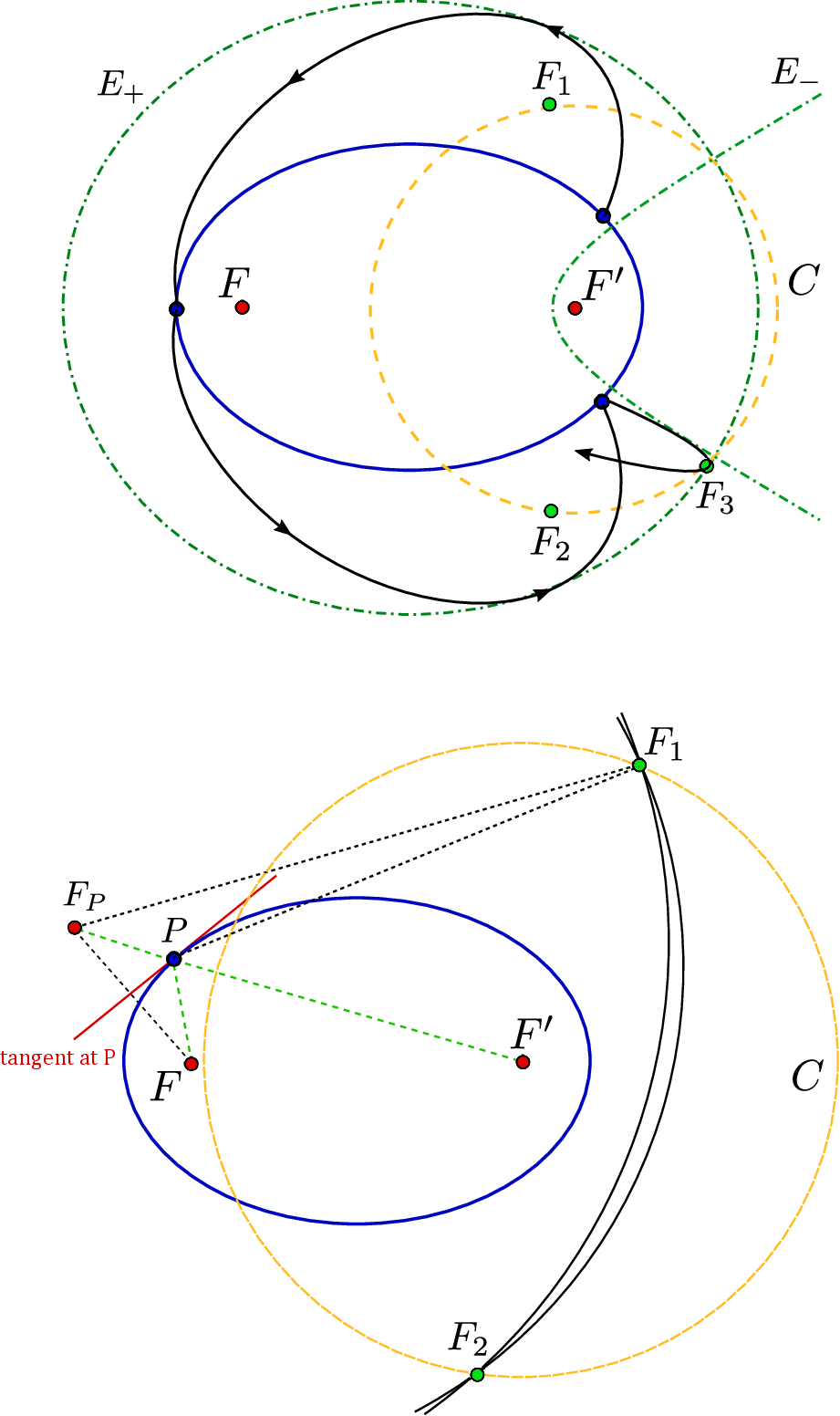}
    \caption{\textit{Top:} Consecutive flight ellipses with reflection along elliptic boundary and envelope curves $E_\pm$ (dash-dotted). \textit{Bottom:} Geometric construction of consecutive foci $F_i$ lying upon the same circle $C$ (dashed) around $F'$.}
    \label{fig:Ellipse_Foci}
\end{figure}





Note that the construction procedure by intersecting two circles to obtain consecutive foci also works for other boundaries. Thus in general consecutive foci $F_i$ can be obtained numerically via a simple algorithm.

Next we provide another proof of the above theorem which illustrates the geometrical meaning of the Gallavotti-Jauslin first integral $D$ as discussed in \cite{takeuchi2021conformal}. This directly covers the hyberbolic boundary case and the cases of combinations of confocal conic sections as well.

\begin{proof}
Let $L$ be the angular momentum,
$A=(A_1,A_2)$ the Laplace-Runge-Lenz vector, $m$ the mass of the particle, $E$ its energy and $\alpha<0$ the coupling constant of the Kepler potential which we assume to be attractive. It has been shown [Lem 2, \cite{takeuchi2021conformal}] that for the elliptic boundary  the quantity
$$D=L^2-2{\bf{c}_K} A_1$$
is a conserved quantity. As in \cite{Felder_2021}, (at least) in the case of $\alpha<0$, we may use the classical relationship 
$$A_1^2+A_2^2=m^2\alpha^2+2 m E L^2$$
to write $D$ as
$$D=\dfrac{A_1^2+A_2^2}{2 m E}-2{\bf{c}_K} A_1-\dfrac{m\alpha^2}{2E}.$$
The normalized vector $\dfrac{1}{m E}\vec{A}$ is the vector connecting the origin to the second focus of the elliptic orbit, i.e. it is equal
to $\vec{F}_i$.

By completing the square in the above {equation} we have
$${(\dfrac{A_1}{mE}-2{\bf{c}_K})^2}+(-\dfrac{A_2}{m E})^2=4{\bf{c}_K}^2+\dfrac{\alpha^2}{E^2}+\dfrac{2 D}{mE}.$$
This means that the second focus $F_i$ of the reflected elliptic orbits lie on a circle centered at the second focus $(2{\bf{c}_K},0)$ of the elliptic reflection wall, with radius given by 
$$R=\sqrt{4{\bf{c}_K}^2+\dfrac{\alpha^2}{E^2}+\dfrac{2 D}{mE}}.$$

\end{proof}

Changing signs of $\alpha$ and $E$ in the above argument extends the result to hyperbolic attractive and repulsive Kepler arcs as well.

As a last analysis for this section we derive the equations of the envelope curves $E_\pm$ of the flight ellipses. In our setting, these envelope curves are the analogues to the caustics  of standard billiards in the plane.

\begin{corollary}\label{cor:Envelopes_Kepler}
The flight ellipses $\mathcal{K}_i$ are confined by the boundary itself and two envelope curves $E_\pm$ given by confocal conics with the same foci $F$ and $F'$ as the boundary $\mathcal{B}_K$. 
The equations for the envelopes read 
$$E_\pm :\frac{\left(x-\bf{c}_K\right)^2}{\left(a\pm\frac{R}{2}\right)^2}+\frac{y^2}{\left(a\pm \frac{R}{2}\right)^2-\bf{c}_K^2}=1.$$
Here as before $R$ is the radius of the foci curve, $\bf{c}_K$ is the linear eccentricity of the conic section boundary and $a$ is the conserved semi-major axis of the Kepler orbit $\mathcal{K}_i$.
\end{corollary}

\begin{proof}
    The proof follows by a direct computation of the envelope curve. For this each Kepler orbit $\mathcal{K}_i$ can be seen as the set of points $(x,y)\in \mathbb{R}^2$ satisfying
    $$\mathcal{K}_i:K(x,y,\alpha_i)=16a^2(x^2+y^2)-(F_{x,i}^2+F_{y,i}^2-4a^2-2xF_{x,i}-2yF_{y,i})^2=0,$$
    with $F_{x,i}=2{\bf{c}_K}+R\cos(\alpha_i)$ and $F_{y,i}=R\sin(\alpha_i)$, where $\alpha_i$ represents the angle labelling the position of the second focus $F_i$. The corresponding envelope curves $E_\pm$ then follows by solving the two equations
    \begin{align*}
        K(x,y,\alpha_i)&=0,\\
        \partial_{\alpha_i}K(x,y,\alpha_i)&=0.
    \end{align*}For a graphical representation of the confocal envelope curves in the elliptic boundary scenario see e.g. Fig. \ref{fig:Ellipse_Foci} \textit{Top}.
     
      Geometrically, if one considers the explicit form of $E_\pm$, one can interpret these curves with the "string construction" of ellipses/hyperbolas with foci $F=(0,0)$ and $F'=(2{\bf{c_k}},0)$ and string of length $|2a\pm R|$ respectively. Here the set of extremal points $A$ on $\mathcal{K}_i$ forming $E_+$ are those points where the string is aligned with the direction of $\overline{F'F_i}$. Similarly, the set of points $B$ forming $E_-$ are points on $\mathcal{K}_i$ where the string is aligned in the opposite direction of $\overline{F'F_i}$ (compare Fig \ref{fig:String_Construction}). In particular this shows that each Kepler orbit $\mathcal{K}_i$ contains exactly two extremal points $(A,B)$ with respect to the string constructions, with $A\in E_+$ and $B\in E_-$.
\end{proof}

\begin{figure}
    \centering
    \includegraphics[scale=0.8]{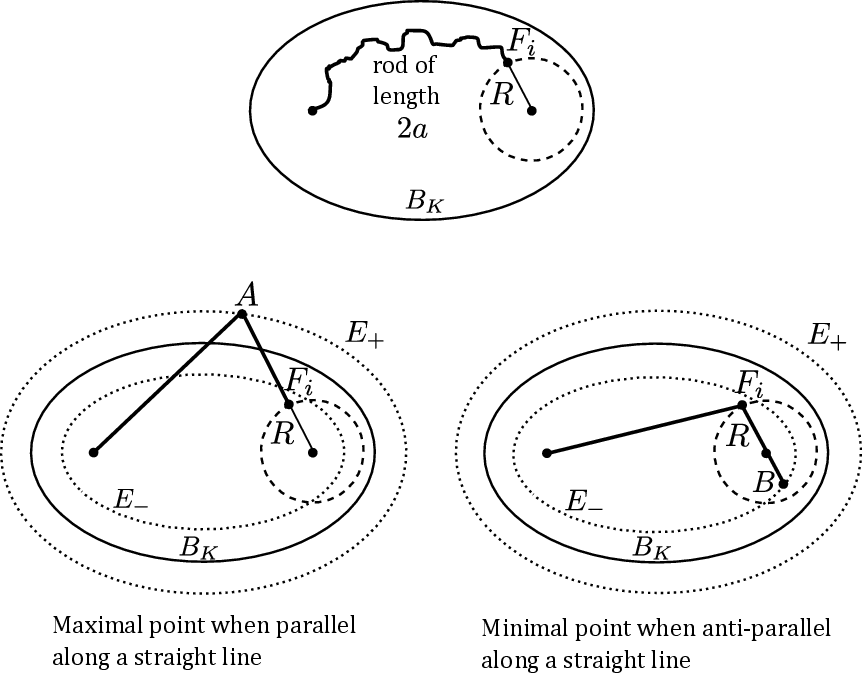}
    \caption{Representative "string construction" for the two envelope curves corresponding to two confocal ellipses.}
    \label{fig:String_Construction}
\end{figure}

Depending on the choice of $a$, $R$ and ${\bf{c}_K}$, the envelope curves correspond to confocal ellipses or hyperbolic arcs. The result is 
similar to the envelope curves obtained for the straight line system as discussed in \cite{gasiorek2023periodic}. Note that for the hyperbolic boundary case in the limit ${\bf{c}_K}\rightarrow \infty$ the system is equivalent to a particle bouncing inside a parabolic mirror with constant gravitational force as discussed e.g. in \cite{jaud2022gravitational}.

\subsection{Focal reflection property}

In this section we consider the billiard dynamics of a Kepler ellipse $\mathcal{K}_i$, whose orbit is passing  through the second focus $F'$ of the boundary $\mathcal{B}_K$. This is the case, when the semi-major axis $a$, the linear eccentricity $\bf{c}_K$ of the boundary and the radius $R$ of the foci circle $C$ satisfy a special condition stated in the Theorem \ref{Thm: Kepler Focal Reflection Property} below.

We first discuss several lemmas.
\begin{lemma}\label{lem:intersection focal ellipses}
Any flight ellipse $\mathcal{K}_i$ intersects $\mathcal{B}_K$ at most at two points.
\end{lemma}
We postpone the proof to Section \ref{sec:Hooke}.

\begin{lemma}\label{lem:perp_orbits}
    Let a Kepler flight ellipse $\mathcal{K}_i$ be degenerate to a line segment and contained in the $x$-axis, with the second focus $F_i=(F_{x,i},0)$ and major axis length $2a=|FF_i|=|F_{x,i}|$.  Then either 
    $$F_{x,i}<\bf{c}_K-\bf{a}_K$$
    or $$F_{x,i}>\bf{a}_K+\bf{c}_K.$$
    The consequent motion of the particle is periodic, with reflections either at the pericenter or the apocenter of {the elliptic boundary} $\mathcal{B}_K$.
\end{lemma}

\begin{proof}
    The two restrictions on $F_{x,i}$ follow from the consideration that part of the Kepler orbit lie outside of the elliptic boundary $\mathcal{B}_K$ along the $x-$axis.
\end{proof}

With this Lemma 
we can now formulate our main theorem of this section, which states the focal reflection property in our setting.
This property is well-known for standard billiards in the plane. Felder's analysis showed a similar phenomenon exist for the line boundary case. The following theorem, stated for the case of a focused elliptic boundary, provides a generalization. As we will see, only one of the two degenerate orbits exists with the particular parameter setting in the following theorem, which plays the role of the nodal point as in \cite{Felder_2021}.

We call a Kepler flight ellipse \emph{admissible} if $E_+$ lies outside $\mathcal{B}_K$, and if  
$$R>0,~ 2a=R+2\bf{c}_K$$
holds, which implies in particular that $E_-$ is degenerated into the line segment $FF'$ in $\mathcal{B}_K$. Clearly if a Kepler flight ellipse is admissible, then so is any other Kepler flight ellipse on its billiard orbit. All of them intersect $\mathcal{B}_K$ at two points, which permit the definition of a complete forward billiard orbit starting from any of them. The condition $R>0$ is included to permit non-stationary billiard orbits to exist.

\begin{theorem}[Focal Reflection Property]\label{Thm: Kepler Focal Reflection Property}
All admissible flight Kepler ellipses pass through $F'$.
Furthermore, if an admissible flight Kepler ellipse does not degenerate into a line segment lying on the designated $x-$axis through $F$ and $F'$, then along the billiard orbit, the flight Kepler ellipses asymptotically converge to the degenerate line segment between $F$ and the point $P_{max}=(0,R+2{\bf{c}_K})$.

\end{theorem}

\begin{proof}
    
A flight Kepler ellipse $\mathcal{K}_i$ is admissible if in particular $2a=R+2\bf{c}_K$ holds, where $R=|F'F_i|$. This condition is equivalent to
$$2a=|FF'|+|F'F_i|.$$
Thus $\mathcal{K}_i$ passes through $F'$.

We now prove the second statement of the theorem.  
We take the $F F'$-axis as the reference axis equal to the $x-$axis. The point $P_{max}$ lies on this axis. We label the position of $F_i$ by the angle $\alpha_i=\angle P_{max} F' F_i$. Similarly we label the reflection point $P_i$ of $\mathcal{K}_i$ (assumed clock-wise oriented when non-degenerate) with $\mathcal{B}_K$ by the angle $\lambda_i=\angle P_{max} F' P_i$ (see Fig. \ref{fig:Focal_Reflection_Property}).
  
We assume that $0<\alpha_i<2 \pi$. We first observe that always $0<\lambda_i<\alpha_i<\lambda_{i+1}$.
This follows from the string construction used in the proof of Corollary \ref{cor:Envelopes_Kepler}. Indeed it follows by this construction that the two intersection points of $\mathcal{K}_i$ with $\mathcal{B}_K$ lies on the different sides of the line $F'F_i$ which realizes the two extremal points as depicted in Fig. \ref{fig:String_Construction}.

Consequently the sequence $\{\alpha_i\}_{i=1,2,3\cdots}$ is strictly decreasing. Moreover, we have the relationship
$$\alpha_{i+1}=\alpha_i-2(\alpha_i-\lambda_i),$$
which follows from the mirror construction of $F_{i+1}$ from $F_i$ as described in Lemma \ref{lemma:Circle_P}.

We now show $\lim_{i \to \infty} \alpha_i=0$. By monotonicity there exists $0\le \alpha^*< 2 \pi$ such that $\lim_{i \to \infty} \alpha_i=\alpha^*$. Consider the limit of the flight ellipses, which by construction has only one intersection with $\mathcal{B}_K$ and is therefore a degenerate Kepler ellipse, or a line segment, which is fixed by the billiard mapping, so it has to lie in the major axis of $\mathcal{B}_K$, thus $\alpha^* \in \{0, \pi\}$. For the case $\alpha^*=\pi$ we see that the degenerate flight ellipse does not passes through $F'$ and cannot be obtained as a limit of a sequence of Keplerian flight ellipses who do so. Thus the only possibility is $\alpha^*=0$ and the limiting Keplerian ellipse is given by the degenerate line segment $F P_{max}$. 

\end{proof}

We remark that similar statement holds for a focused hyperbolic branch boundary not enclosing the Kepler center. Moreover, this focal reflection property holds true for a parabolic mirror since this can be seen as the limiting case of the elliptic/hyperbolic boundary. In the limiting case ${{\bf{c_K}} \to \infty}$, one obtains the motion in a focused parabolic mirror under the influence of a constant gravitational force as studied e.g. in \cite{Masalovich, jaud2022gravitational}. 
In the centered circle boundary case, orbits are straight lines and the system is highly symmetric. The focal reflection property breaks down.

\begin{figure}[!htb]
    \centering
    \includegraphics[scale=0.55]{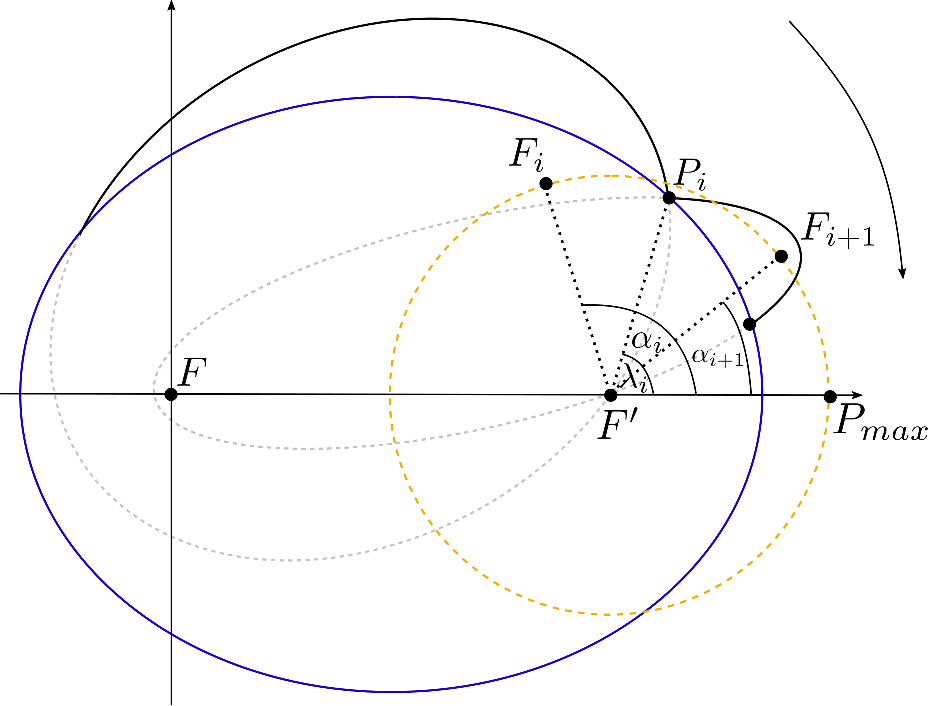}
    \caption{Graphical representation for the asymptotic behaviour of Kepler ellipses "passing" theoretically through $F'$.}
    \label{fig:Focal_Reflection_Property}
\end{figure}


\subsection{The family of directrices in the Kepler billiard system}

To conic sections with a preferred focus, we can associate a directrix $d$. The family of Kepler ellipses $\mathcal{K}_i$ with the preferred focus at the origin gives rise to a family of directrices $d_K$. 

In this section, we investigate the geometry of the directrices in our integrable Kepler billiards. The result can be extended to various other cases, such as orbits of the attractive/repulsive Kepler problem inside  $\mathcal{B}_K$.

\begin{lemma}\label{lem:directrices}
Let the coordinate of the second Kepler focus be 
$$F_i=(F_x,F_y)={(2{\bf{c}_K}+R\cos(\alpha_i),R\sin(\alpha_i))}.$$
Then the directrix $d_K$ of an Kepler 
elliptic orbit $\mathcal{K}_i$ with the second focus $F_i$ associated to the angle $\alpha_i$ is given by:
\begin{equation}\label{eq:directrices_Kepler}
    d_K:= x\cdot F_x+y\cdot F_y +2a^2-\frac{1}{2}\cdot (F_x^2+F_y^2)=0.
\end{equation}
\end{lemma}

\begin{proof}
Let $\mathcal{K}_i$ be a Kepler flight ellipse with semi-major axis $a$, with the preferred focus $F$ at the origin and $F_i$ as defined above. The distance of the closest point $P\in d_K$ from the origin is given by $|OP|=\frac{a}{e}-ae,$ where $e$ is the excentricity, related to the foci distance by $ae=\frac{1}{2}| \vec{F}_i|$. 
Thus the point $P\in d_K$ has the expression
$$\vec{P}=-\frac{\vec{F}_i}{|\vec{F}_i|}\cdot \left(\frac{a}{e}-ae\right)=-\frac{\vec{F}_i}{|\vec{F}_i|^2}\cdot \left(2a^2-\frac{1}{2}|\vec{F}_i|^2\right).$$
The directrix has to run perpendicular to $\vec{F}_i$ through $P$, i.e. it has to take the form
$$d_K:\vec{x}=\begin{pmatrix}
        x\\
        y
\end{pmatrix} =\vec{P}+\lambda \begin{pmatrix}
        -F_y\\
        F_x
\end{pmatrix},~~\lambda \in \mathbb{R},$$
    with the point $P$ on it. Plugging in the expression of $P$ we determine $\lambda$ and consequently the expression for $d_K$.
\end{proof}

Along a billiard orbit, the family of directrices has an envelope. By an \emph{envelope curve} we mean a curve which contains the envelop of the family of directrices along a billiard orbit with the first integrals being fixed. We shall obtain such a curve from the foci curve. Depending on the relation between $2{\bf{c}_K}$ and $R$, this curve is given by different conic sections.

\begin{theorem}
If $2{\bf{c}_K}\neq R$, then the envelop of the family of directrices is contained in either an ellipse or a hyperbola. The equations for these two conic sections are
$$\frac{4(R^2-2{\bf{c}_K})^2\cdot \left(x+{\bf{c}_K}\cdot \frac{4{\bf{c}_K}^2-R^2-4 a^2}{R^2-4{\bf{c}_K}^2}\right)^2}{R^2\cdot (4 a^2+4{\bf{c}_K}^2-R^2)^2}+\frac{4(R^2-4{\bf{c}_K}^2)\cdot y^2}{(4 a^2+4{\bf{c}_K}^2-R^2)^2}=1.$$

For $2{\bf{c}_K}=R$, i.e. when the foci circle $C$ is passing through $F$, the envelope curve for the family of directrices is the parabolic arc with the equation
$$x=\frac{{\bf{c}_K}}{2a^2}\cdot y^2-\frac{a^2}{2{\bf{c}_K}}+2{\bf{c}_K}.$$
For all cases one focus of these conic sections lies at $F'(2{\bf{c}_K},0)$, i.e. the second focus of $\mathcal{B}_K$. The second focus for case $2{\bf{c}_K}\neq R$ lies at $\left(-\frac{8a^2{\bf{c}_K}}{4{\bf{c}_K}^2-R^2},0\right)$.
\end{theorem}

\begin{proof}
The envelope curve for the family of directrices we look for is determined by
    \begin{align*}
        d_K(x,y,\alpha_i)&=0,~~~~(I)\\
        \partial_{\alpha_i} d_K(x,y,\alpha_i)&=0.~~~~(II)
    \end{align*}
The explicit expression for $d_K$ is given by    $$d_K=x(2{\bf{c_K}}+R\cos(\alpha_i))+yR\sin(\alpha_i)+2 a^2-2{\bf{c_K}}^2-\frac{R^2}{2}-2{\bf{c_K}}R\cos(\alpha_i).$$
Condition $(II)$ thus reads

$$-xR\sin(\alpha_i)+yR\cos(\alpha_i)+2{\bf{c_K}}R\sin(\alpha_i)=0.$$

and thus

$$\tan(\alpha_i)=\frac{\sin(\alpha_i)}{\cos(\alpha_i)}=\frac{y}{x-2{\bf{c}_K}}.$$

We write
\begin{align*}
        \sin(\alpha_i)&= \frac{\epsilon\tan(\alpha_i)}{\sqrt{1+\tan^2(\alpha_i)}},\\
        \cos(\alpha_i)&=\frac{\epsilon}{\sqrt{1+\tan^2(\alpha_i)}},
\end{align*}
where $\epsilon =\pm 1$ depending on the quadrant $\alpha_i$ lies in.
Plugging in these expressions in $(I)$ together with the result for $\tan(\alpha_i)$ obtained from $(II)$ we obtain:
\begin{align*}
(I) &\Leftrightarrow \frac{R\epsilon (x-2{\bf{c}_K})}{\sqrt{1+\tan^2(\alpha_i)}}+\frac{R\epsilon y \tan(\alpha_i)}{\sqrt{1+\tan^2(\alpha_i)}}+2x{\bf{c}_K}+\frac{1}{2}(4a^2-4{\bf{c}_K}^2-R^2)=0\\
& \Leftrightarrow \frac{R\epsilon[(x-2{\bf{c}_K})^2+y^2]}{\sqrt{(x-2{\bf{c}_K})^2+y^2}}=-2x{\bf{c}_K}-\frac{1}{2}(4a^2-4{\bf{c}_K}^2-R^2).
\end{align*}

Squaring the above expression with $\epsilon^2=1$ we get
$$R^2[(x-2{\bf{c}_K})^2+y^2]=[2x{\bf{c}_K}+\frac{1}{2}(4a^2-4{\bf{c}_K}^2-R^2)]^2.$$

This is equivalent to
    
$$x^2(R^2-4{\bf{c}_K}^2)-2{\bf{c}_K}x(4a^2-4{\bf{c}_K}^2+ R^2)+4{\bf{c}_K}^2R^2-\frac{1}{4}(4a^2-4{\bf{c}_K}^2-R^2)^2+R^2y^2=0,$$
from which the results in different cases follow. If $R^2-4{\bf{c}_K^2}=0$, we obtain the parabolic envelope. Otherwise we obtain elliptic/hyperbolic envelopes described above. 
\end{proof}

Note that in the large $\bf{c_K}$ case the hyperbolic boundary turns into a parabolic boundary. In this limit holds that $a\approx {\bf{c_K}}\gg R$ and one can perform a limit calculation for the envelope of directrices in this szenario to obtain
$$\frac{4{\bf{c_K}}^2x^2}{R^2({\bf{c_K}}^2+a^2)}-\frac{4{\bf{c_K}}^2y^2}{{\bf{c_K}}^2+a^2}=1.$$
This means that in the limit of a parabolic boundary the envelope of directrices is always given by a hyperbola.


\section{Hooke billiards}\label{sec:Hooke}
\subsection{The Hooke-Kepler correspondence}
In two dimensions, the Kepler system is dual to the Hooke system. This and other dualities of power-law forces were first discovered by \cite{MacLaurin} and later on discussed in \cite{Goursat}. In particular by identifying the plane with the complex plane it is seen \cite{Goursat,MacLaurin} that under the conformal mapping $z\mapsto z^2$, unparametrized trajectories 
of the Hooke system are mapped to unparametrized trajectories 
of the Kepler problem. In particular, the centered Hooke ellipses are mapped to focused Kepler ellipses. 
Moreover, it is not hard to check that both foci of the Hooke ellipse  are mapped to the second focus $F'$ of the Kepler ellipse: Say the Kepler ellipse is given by $$|q|+|q-F'|={C_1}$$
for a constant $C_1$. This means the pre-image satisfies 
\begin{equation}\label{eq: square square}
|z^2|+|z^2-F'|={C_1}.
\end{equation}
Let $\sqrt{F'} \in \mathbb{C}$ be any square root of $F'$.
We have 
$$|z+\sqrt{F'}|^2+|z-\sqrt{F'}|^2=2 (|z|^2+|F'|)$$
and 
$$|z+\sqrt{F'}| \cdot |z-\sqrt{F'}|= |z^2-F'|.$$
Thus we obtain from \eqref{eq: square square} that
$$(|z+\sqrt{F'}|+|z-\sqrt{F'}|)^2={C_2} $$ 
for some constant $C_2$. Thus the pre-image is an ellipse with foci $\pm \sqrt{F'}$.

Since the mapping is conformal, this correspondence extends to the corresponding billiard systems \cite{Panov}, \cite{takeuchi2021conformal}. The transformation also leads to many new classes of integrable mechanical billiards in systems similar to the Stark problem (Kepler problem with an external constant force field)\cite{takeuchi2021conformal}. 

Due to the linearity of the force in the Hooke system, as a rule, it is often easier to carry out investigations in the Hooke system than in the Kepler counterpart. Some questions concerning the Kepler problem become much apparent after being transformed into the Hooke problem. As an example we first prove Lemma \ref{lem:intersection focal ellipses}.

\begin{proof}
(of Lemma \ref{lem:intersection focal ellipses}): The pre-images of both the flight ellipse and the boundary are centered ellipses. They intersect at most at four points, coming in pairs symmetric with respect to the center. Thus the image of the set of intersection points consists at most of two points.
\end{proof}

As will be shown in the following, the computation of foci curve in the Kepler billiard here is one of the few examples, in which the calculations can be carried out in a much easier way in the Kepler billiard. The foci curve of the Hooke billiard is thus obtained in a simple fashion. It does not seem to be easy to be directly identified in the Hooke billiard though, as the curve is in general of degree 4.

\subsection{The foci curve and the directrices}
The integrability of Hooke billiards inside 
a centered conic section domain is well-known \cite{Jacobi1866}, \cite{Fedorov2001}. Here we illustrate the geometric meaning of an additional first integral of these Hooke billiards.

We consider orbits of the attractive Hooke problem, which are centered ellipses $\mathcal{H}_i$ in the plane. As in the case of Kepler ellipses it turns out useful to describe these ellipses at a given energy $E$ in terms of their foci, which lie symmetrically with respect to the Hooke center that we put at the origin. 

We denote one of the two foci by $F_i=(F_{x,i},F_{y,i})$. The Hooke flight ellipse at given positive energy $E$ with a focus at $F_i$ is thus given by the equation:

$$\mathcal{H}_i:\left(x^2+y^2-\frac{2E}{k}\right)^2=\left[(x-F_{x,i})^2+(y-F_{y,i})^2\right]\cdot \left[(x+F_{x,i})^2+(y+F_{y,i})^2\right].$$

Note that in contrast to the Kepler system, for which the semi-major axis $a$ of the Kepler ellipse is conserved, the semi-major axis of the Hooke ellipse (denoted by $a_{H}$) does not relate in the same way to the total energy. It varies with the position of the focus according to
$$a_{H}=\sqrt{\frac{1}{2}\cdot (F_{x,i}^2+F_{y,i}^2)+\frac{E}{k}}.$$

We can now state the first main theorem regarding the integrable Hooke billiard system.

\begin{theorem}[Hooke foci curve]
    Consider the Hooke billiard with a centered conic section $\mathcal{B}_H$ as boundary. Then the foci of a Hooke flight elliptic orbits lie on a foci curve given by a Cassini oval of the form
    $$[(x-{\bf{c}_H})^2+y^2]\cdot [(x+{\bf{c}_H})^2+y^2]=R_H^4,$$
    where ${\bf{c}_H}$ is the linear eccentricity of $\mathcal{B}_H$ and $R_H$ can be regarded as a generalised radius. 
\end{theorem}

We denote the set of points $(x,y)\in \mathbb{R}^2$ lying on the Cassini curve by $\mathcal{C}({\bf{c}_H},R_H)$.

\begin{proof} The foci curve in the Hooke system is the pre-image under the complex square mapping $\mathbb{C} \mapsto \mathbb{C}, \quad z \mapsto q=z^2$ of the foci curve of the associated Kepler system. 
The pre-image of the Kepler foci curve $|q-2{\bf{c}_K}|=R$, which is a circle with radius $R$ and centered at $(2{\bf{c}_K},0)$, is:
$$|z^2-2{\bf{c}_K}|=R
$$    
and thus

$$|z-\sqrt{2{\bf{c}_K}}|\cdot |z+\sqrt{2{\bf{c}_K}}|=R,$$  

which describes a Cassini oval. In real coordinates $z=x+iy$ the equation reads
\begin{equation*}
    [(x-{\bf{c}_H})^2+y^2]\cdot [(x+{\bf{c}_H})^2+y^2]=R_H^4,
\end{equation*}
with ${\bf{c}_H}=\sqrt{2{\bf{c}_K}}$ 
and generalised radius $R_H=\sqrt{R}$.
\end{proof}
 
The eccentricity of the Cassini oval is defined as
\begin{equation*}
    e:=\frac{R_H}{\bf{c}_H}.
\end{equation*}
This parameter helps to distinguish the different shapes of the Cassini ovals.

Figure \ref{fig:Cassini} illustrates the obtained Cassini ovals for the different possibilities of the Cassini eccentricity $e$ compared with the corresponding circles in the Kepler system and the elliptic boundary.

\begin{figure}[htb]
    \centering
    \includegraphics[scale=0.68]{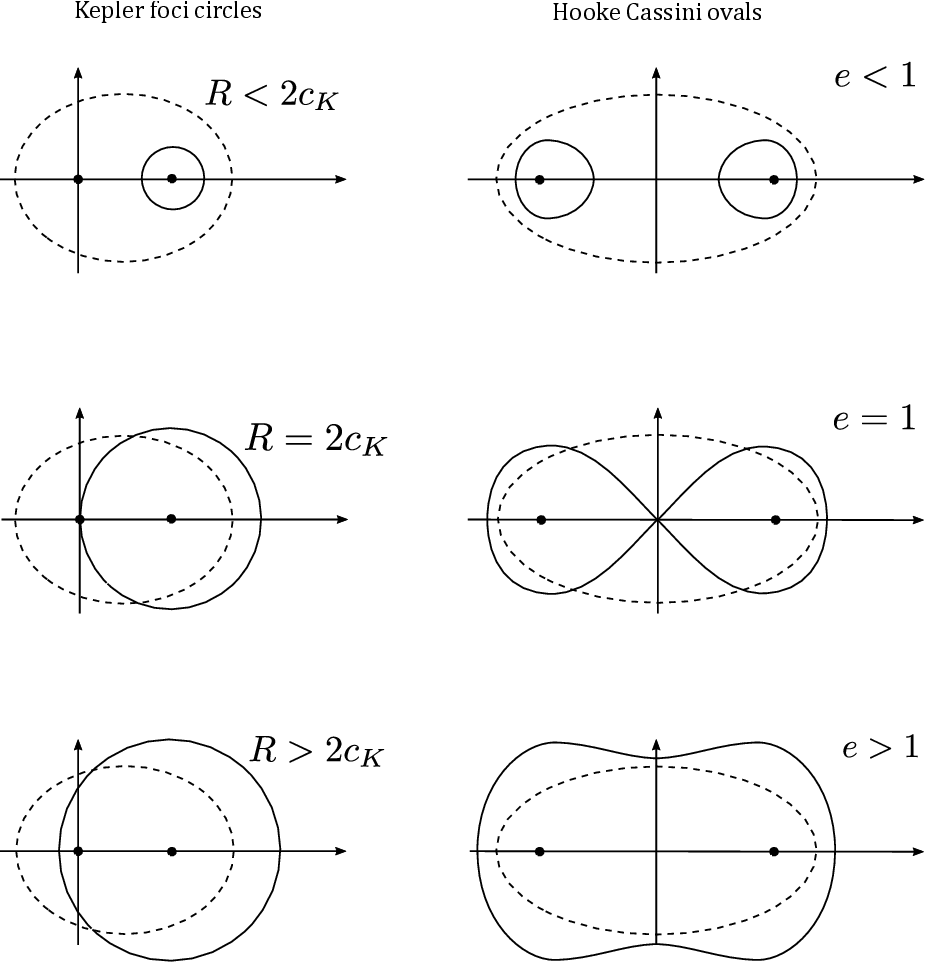}
    \caption{Illustration of Cassini ovals (right) for different values of their corresponding eccentricity $e$ and the related foci circles in the Kepler systems (left). The dashed lines correspond to the associated elliptic boundaries.}
    \label{fig:Cassini}
\end{figure}

According to his son Jacques, Giovanni Domenico Cassini proposed that the Sun should move along one of these Cassini ovals around the earth, which occupies one of the Cassini foci \cite{Cassini1740}.

The corresponding envelope curves associated to this integrable Hooke billiard system in the plane, as have been established in multi-dimensional case in \cite{Fedorov2001} can as well be obtained in the planar case directly from the Kepler case as pre-images under the complex square mapping. We have

\begin{corollary}\label{cor:Env_Hooke}
    The envelope curves $E_{H\pm}$ in the integrable Hooke billiards with an elliptic boundary with foci $F_H$ and $-F_H$ are given by confocal ellipses/hyperbolas with equations:
    
    $$E_{H\pm}=\frac{x^2}{\frac{1}{2}\cdot ({\bf{c}_H}^2\pm R_H^2)+\frac{E}{k}}+\frac{y^2}{\frac{1}{2}\cdot (-{\bf{c}_H}^2\pm R_H^2)+\frac{E}{k}}=1,$$
    where ${\bf{c}_H},R_H$ are the parameters of the Cassini foci curve $\mathcal{C}({\bf{c}_H},R_H)$ and $E$ is the energy of the Hooke system.
    \end{corollary}

In Fig. \ref{fig:Cassini_Hooke} one finds an illustration of the associated Hooke billiard with an elliptic boundary and envelope curve. As shown the consecutive foci $F_i$ for the Hooke flight orbits lie along the Cassini oval.

\begin{figure}[htb]
    \centering
    \includegraphics[scale=0.7]{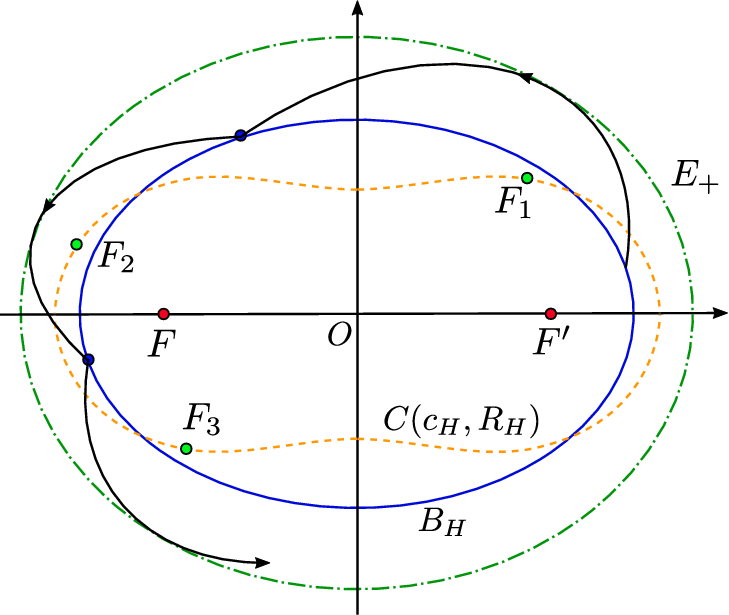}
    \caption{Cassini foci curve $\mathcal{C}({\bf{c}_H},R_H)$ (dashed), associated flight ellipses and envelope $E_+$ (dash-dot) for ${\bf{c}_H}=1$ and $R_H=\sqrt[4]{2}$.}
    \label{fig:Cassini_Hooke}
\end{figure}

Note that in the limiting case where $R_H\gg {\bf{c}_H}$, i.e. $e\gg 1$, the foci curve resembles a circle centered at the origin.




\subsection{The focal reflection property}

Similar to the focal reflection property {for the} Kepler system we can state an equivalent property for the Hooke system.

We state the result in relation to the Kepler case. A Hooke flight ellipse is called admissible if its image is admissible in the corresponding Kepler system. In particular, geometrically an admissible Hooke ellipse passes through both foci of the elliptic boundary $\mathcal{B}_H$. This is the case if $\left({\bf{c_H}}-\frac{2E}{k}\right)^2=R_H^4$ is satisfied and the orbit is not degenerated to a straight line.

\begin{theorem}[Hooke Focal Reflection Property]
Let $F_H$ and $F_H'$ be the symmetric foci of the conic boundary $\mathcal{B}_H$ of the Hooke billiard. 
If one Hooke ellipse $\mathcal{H}_1$ passes through $F_H$ and $F_H'$ then all consecutive Hooke  
ellipse $\mathcal{H}_i$ pass through these points. Asymptotically the Hooke ellipses tend to the axis of foci $F_H$ and $F_H'$.    
\end{theorem}

\begin{proof}
    This is a direct consequence of the Focal Reflection Property, Thm. \ref{Thm: Kepler Focal Reflection Property} in the Kepler system. The Kepler conics passing through the second focus of the boundary are pulled back by $z \mapsto q=z^2$ to Hooke conics passing through both foci of the boundary. Without loss of generality we may assume that the second focus of the Kepler boundary is on the positive real-axis. The foci of the Hooke boundary thus lie on the real axis, symmetric with respect to the origin. The foci of the flight ellipses thus converge to the pre-image of the point $P_{max}$ which also lie on the real axis. In particular, they tend to the real axis.
\end{proof}

\subsection{The family of directrices in the Hooke system}

For the Kepler system, we have derived the envelopes for the family of directrices and have seen that these correspond to (parts of) conic sections with a focus at the focus $F'$ of the boundary $\mathcal{B}_K$. 

For the Hooke system we now carry out the same steps. Here we simply present a parametric expression for the envelope curves. These in general are algebraic curves of degree 4. Note that the directrices in the Hooke and Kepler system are not directly related by the complex square mapping, which does not preserve straight lines.

We aim to describe the foci of the Hooke flight ellipses $F_i\in \mathcal{C}({\bf{c}_H},R_H)$ in terms of $\bf{c}_H$ and $R_H$. In polar coordinates, we have $F_i=(r_\pm(\varphi),\varphi)$, where for the radius we have two different choices

$$r_\pm(\varphi)=\sqrt{{\bf{c}_H}^2 \cos(2\varphi)\pm \sqrt{{\bf{c}_H}^4\cos^2(2\varphi)+R_H^4-{\bf{c}_H}^4}}.$$
corresponding to the 2 possible intersections of the ray with angle $\varphi$ with $\mathcal{C}({\bf{c}_H},R_H)$.

The directrices of the Hooke ellipses are found by a direct computation, and are given by

\begin{equation*}
    d_H:=x\cdot F_x+y\cdot F_y +\frac{1}{2}\cdot \left(F_x^2+F_y^2\right)+\frac{E}{k}=0,
\end{equation*}
with $F_x=r_\pm(\varphi)\cdot \cos(\varphi)$ and $F_y=r_\pm(\varphi)\cdot \sin(\varphi)$.

\begin{figure}
    \centering
    \includegraphics[scale=0.7]{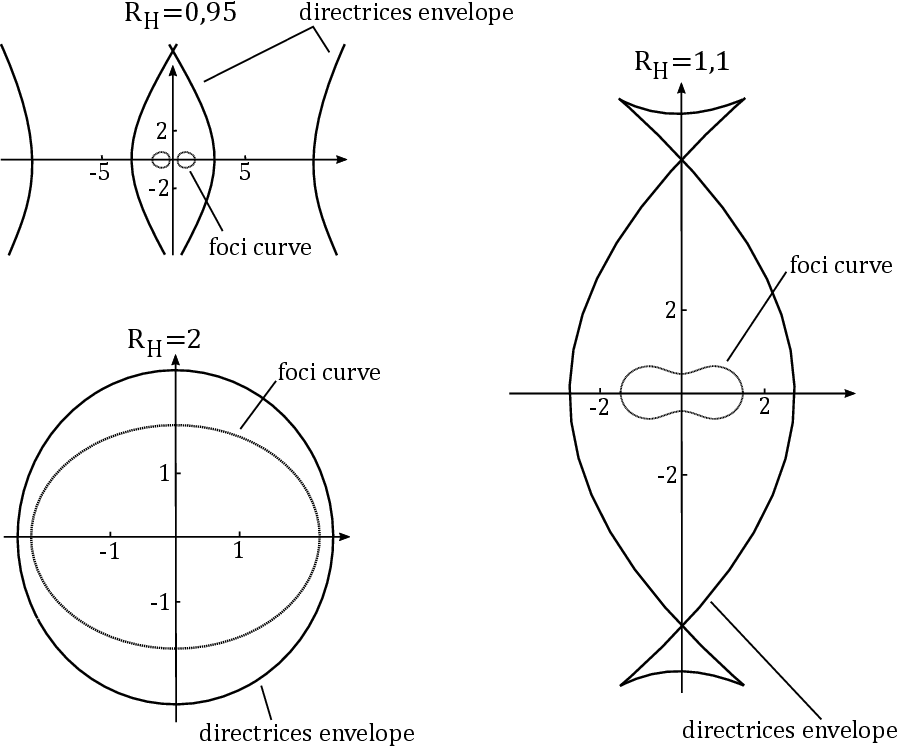}
    \caption{Envelope curves for family of directrices for $R_H\in\{0.95,1.1,2\}$ at fixed ${\bf{c}_H}=1$ and $E/k=3$.}
    \label{fig:Hooke_Direct}
\end{figure}

\begin{theorem}
    The envelope curve of the family of directrices of Hooke flight ellipses {with ${\bf{c_H}}>0$ are given in parametric form by}
    \begin{align*}
    x(\varphi)&=-\left(\frac{E}{k\cdot (F_x^2+F_y^2)}-\frac{1}{2}\right)\cdot \partial_\varphi F_y-F_x,\\
    y(\varphi)&=\left(\frac{E}{k\cdot (F_x^2+F_y^2)}-\frac{1}{2}\right)\cdot \partial_\varphi F_x-F_y.
\end{align*}
For ${\bf{c_H}}=0$, \emph{i.e.} when $\mathcal{B}_H$ is a circle, the envelope of the directrices is given by the circle 
$$x^2+y^2=\frac{2E}{k}.$$
\end{theorem}
\begin{proof}
    The envelope curve is determined by
    \begin{align*}
        d_H(x,y,\varphi)=&0,~~~~(I)\\
        \partial_\varphi d_H(x,y,\varphi)&=0.~~~~(II)
    \end{align*}
We have
    $$d_H(x,y,\varphi)=x\cdot F_x(\varphi)+y\cdot F_y(\varphi) +\frac{1}{2}\cdot \left(F_x(\varphi)^2+F_y(\varphi)^2\right)+\frac{E}{k}.$$
Then by condition $(II)$ it follows that
    $$\partial_\varphi d_H=x\cdot \partial_\varphi F_x+y\cdot \partial_\varphi F_y+F_x\cdot \partial_\varphi F_x+F_y\cdot \partial_\varphi F_y=0.$$
With $\vec{z}=(x, y),~ \vec{F}=(F_x,F_y)^T$ this condition can be alternatively written as
    $$[\vec{z}+\vec{F}(\varphi)]\cdot \partial_\varphi \vec{F}(\varphi)=0.$$
The first solution to this equation is given by:
    \begin{align*}
        x(\varphi)&=-\mu \cdot \partial_\varphi F_y -F_x,\\
        y(\varphi)&=\mu \cdot \partial_\varphi F_x -F_y,
    \end{align*}
where $\mu$ is a scaling parameter, which can then be uniquely determined by $(I)$. 
Indeed, plugging these two expressions into $(I)$, we have

    $$\mu(F_y\cdot \partial_\varphi F_x-F_x\cdot \partial_\varphi F_y)-\frac{1}{2}(F_x^2+F_y^2)+\frac{E}{k}=0.$$
It is not hard to straightforwardly verify that
    $$F_x\cdot \partial_\varphi F_y-F_y\cdot \partial_\varphi F_x=F_x^2+F_y^2.$$
With these we obtain
    $$\mu=\frac{E}{k\cdot (F_x^2+F_y^2)}-\frac{1}{2}.$$
Plugging in this expression for the two coordinates $x$ and $y$  one obtains the result.
    
The second possible solution to the equation $[\vec{x}+\vec{F}(\varphi)]\cdot \partial_\varphi \vec{F}(\varphi)=0$ is given by $\vec{z}=-\vec{F}(\varphi)$. We assume that a continuous part of the envelop curve is given by this second solution, as otherwise the first solution always holds. Inserting this in $(I)$ yields
    $$x^2+y^2=\frac{2E}{k},$$
which is a circle centered at the origin, with radius $\sqrt{2E/k}$.
Since $\vec{z}=-\vec{F}$, the corresponding foci curve is given
$$F_x^2+F_y^2=\frac{2E}{k}.$$
This means that a continuous part of the foci curve is part of a centered circle. This is only possible if $\mathcal{C}({\bf{c_H}},R_H)$ is a circle, which is only possible if ${\bf{c_H}}=0$, \emph{i.e.} if $\mathcal{B}_H$ is a circle.
\end{proof}

In Fig. \ref{fig:Hooke_Direct} one finds various examples for the envelope of directrices for different parameters of $R_H$ at fixed $E/k$ and ${\bf{c}_H}>0$.


Note that when $R_H$ is large, the foci curve asymptotically resembles a circle centered at the origin with radius $R_H$. The corresponding envelope curve for the family of directrices also resembles a concentric circle with radius $r$ given by 

$$r=\frac{E}{k\cdot R_H}+\frac{1}{2}R_H.$$



\section*{Acknowledgement}
Lei Zhao is supported by DFG ZH 605-1/2.

\bibliographystyle{cas-model2-names} 
\bibliography{main.bib}
\end{document}